\documentclass[12pt]{amsart}

\usepackage{amssymb}
\usepackage{pb-diagram}
\usepackage{enumitem} 
\usepackage[dvipsnames]{xcolor}
\usepackage{tikz-cd}
\usepackage{graphicx}
\usepackage{amsmath}
\numberwithin{equation}{section}
\usepackage[all]{xy}

\usepackage{mathrsfs} 
\usepackage[colorlinks=true,pagebackref,hyperindex,citecolor=blue,linkcolor=red]{hyperref}

\newtheorem{teo}{Theorem}[section]
\newtheorem{prop}[teo]{Proposition} 
\newtheorem{cor}[teo]{Corollary} 
\newtheorem{lema}[teo]{Lemma}
\newtheorem{defi}[teo]{Definition} 

\newtheorem{obs}[teo]{Remark}

\newcommand{\KK}{{\omega_X}}

\newcommand{\R}{{\mathcal{R}}}
\newcommand{\E}{{\mathcal{E}}}

\newcommand{\V}{{\mathcal{V}}}
\newcommand{\LL}{{\mathcal{L}}}
\newcommand{\Oc}{{\mathcal{O}}}

\newcommand{\bC}{{\mathbb{C}}}

\DeclareMathOperator*{\Pic}{Pic}
\DeclareMathOperator*{\End}{End}
\DeclareMathOperator*{\Hom}{Hom}
\DeclareMathOperator*{\Grass}{Grass}

\DeclareMathOperator*{\Dim}{dim}
\DeclareMathOperator*{\Deg}{deg}
\DeclareMathOperator*{\Rk}{rk}

\newcommand{\ox}{{\otimes}}

\allowdisplaybreaks

\numberwithin{equation}{section}

\begin{document}

\baselineskip=16pt

\title[Moduli of Higgs coherent systems]{Moduli of Higgs Coherent Systems with stable underlying vector bundles over a curve}

\author{Edgar I. Castañeda-González}

\address{Centro de Investigación en Matemáticas, A.C
\newline  Jalisco \\ Guanajuato, Guanajuato, M\'exico.}
\email{edgar.castaneda@cimat.mx}

\thanks{}

\subjclass[2010]{}


\date{\today}

\begin{abstract}
Let $X$ be a smooth, connected complex projective curve of genus at least $2$. A Higgs coherent system is an augmented bundle $(E,V)$, where $E$ is a holomorphic vector bundle, and $V$ is a linear subspace of the spaces of Higgs bundles of $E$. The purpose of this paper is twofold. First, we introduce the moduli problem of Higgs coherent systems over $X$. Second, we construct the fine moduli space for Higgs coherent systems with stable underlying vector bundles.
\end{abstract}

\maketitle

\section{Introduction}
Let $X$ be a smooth connected complex projective curve of genus $g\geq 2$, and denote by $\KK$ its canonical bundle. The moduli spaces of Higgs bundles and coherent systems over $X$ have been extensively studied due to their relevance in connecting geometry, topology, and physics. 

A Higgs bundle over $X$ is a pair $(E,\varphi)$, where $E$ is a vector bundle over $X$ and $\varphi$ is a Higgs field of $E$, that is $\varphi$ is a section in $H^0(X,\End(E)\ox \KK)$. Higgs bundles were first studied by  N. J. Hitchin as the solutions of the Yang-Mills equations on a Riemann surface \cite{hitchin} establishing a connection between differential equations and algebraic geometry. Since then, they have become central objects intersecting geometry, topology, and physics (\cite{whatHiggs}, \cite{janHiggs}, \cite{lauraHiggs}).  Diverse perspectives include foundational work by  M. F. Atiyah and R. Bott \cite{atiyahbott}, the generalization to higher dimension by C. Simpson \cite{simpson}, the study of its moduli space by T. Hausel \cite{hausel}, and its important role in N. Bao Ch$\hat{\text{a}}$u’s proof of the Fundamental Lemma \cite{ngo}. On the other hand, a coherent systems is a pair $(E,V)$, where $E$ is a vector bundle and $V$ is a linear subspace of $H^0(X,E)$. Coherent systems have contributed significantly to the understanding of the geometry of the moduli space of vector bundles \cite{bradlowsc}, \cite{bradprada}, \cite{newsteadques}.

In this paper, we study augmented bundles that are related to Higgs bundles and coherent systems, called {\bf Higgs coherent systems}. These augmented bundles can be understood as a collection of Higgs bundles approached through the perspective of coherent systems, where we consider a vector bundle together with a linear subspace of Higgs fields rather than a single Higgs field. Specifically, a Higgs coherent system of type $(n,d,k)$ is a pair $(E,V)$ where $E$ is a vector bundle over $X$ of rank $n$ and  degree $d$, and $V$ is a $k$-dimensional subspace of $H^0(X,\End(E)\ox\KK)$. The aim of this paper is to study the moduli problem of Higgs coherent system whose underlying vector bundles are stable and to construct the moduli space of such augmented bundles.

Let $n,d,k$ be three integers such that $(n,d)=1$ and $0<k\leq n^2(g-1)+1$. Denote by $SH(n,d,k)$ the set of Higgs coherent systems $(E,V)$ of type $(n,d,k)$ such that $E$ is a stable vector bundle. The stability assumption allows us to describe the isomorphism classes for these Higgs coherent systems and to construct a parameter scheme for them. This parameter scheme inherits some of geometry properties of the moduli space of stable bundles, as stated in the following result.
\begin{teo}{(Theorem \ref{ParaStable})}
The space of isomorphism classes of Higgs coherent systems in $SH(n,d,k)$ is a smooth irreducible projective scheme of dimension $(k+1)(n^2(g-1)+1)-k^2$.
\end{teo}
We construct a family of Higgs coherent systems in $SH(n,d,k)$ with the universal property parameterized by the above parameter scheme, this allows us to prove the existence of a fine moduli space. The result is presented as follows.
\begin{teo}{(Theorem \ref{ModuliHcsstable})}
There exists a fine moduli space for Higgs coherent systems in $SH(n,d,k)$ in the cathegory of smooth irreducible projective schemes.
\end{teo}
As a corollary of the above result, we obtain the moduli space of Higgs coherent systems of type $(1,d,k)$, previously studied in \cite{oscar}, as established in Corollary \ref{ModuliHcs1dk}.

The paper is organized as follows. In Section \ref{Sec1}, we introduce Higgs coherent systems of type $(n,d,k)$ over $X$ and their basic properties. Then, we introduce the concept of families parameterized by a scheme and set out the problem we are interested in; the moduli problem of Higgs coherent systems of type $(n,d,k)$ on $X$. Our main results are stated and proved in Section \ref{Sec3}. 

{\bf Notation:} We work over the field of complex numbers, denoted by $\bC$. Unless stated otherwise, by scheme we always mean a noetherian projective scheme of finite type over $\bC$; by a point, we always mean a closed point; and by a vector bundle, we always mean a holomorphic vector bundle. Let $X$ and $Y$ be two schemes, we denote by $X\times Y$ the fiber product $X\times_{\text{Spec}(\bC)}Y$. For a vector bundle $E$ over $X$, we denote by $\Deg(E)$ its degree, by $\Rk(E)$ its rank, and by  $h^i(E)$, the dimension of the cohomology group $H^i(X,E)$. We will write $H^i(E)$ instead of $H^i(X,E)$ when no confusion arises. Let $\alpha\colon F\to E$ be a morphism of vector bundles, we denote by $H^i(\alpha)\colon H^i(F) \to H^i(E)$ the induced morphism in cohomology groups. If $x$ is a point in $X$, we denote by $E_x$ the fiber of $E$ at $x$. Similarly, if $\E$ is a vector bundle over a product $X\times Y$ and $y$ is a point in $Y$, we denote by $\E_y$ the vector bundle over $X$ obtained by the restriction of $\E$ to $X\times \{y\}$. Let $f\colon X\to Y$ be a morphism of schemes, for $i\geq 0$, we denote by $R^if_{*}(-)$ the $i$-th right derived functor of the direct image functor $f_*(-)$.

\section{Moduli problem of Higgs coherent systems}\label{Sec1}
Let $X$ be a smooth connected complex projective curve of genus $g\geq 2$, and $\KK$ its canonical bundle. In this section, we introduce Higgs coherent systems on $X$ and study their main properties. Subsequently, we introduce the concept of families parameterized by a scheme and set out the problem of Higgs coherent systems on $X$.

\subsection{Higgs coherent systems}
Let $n,d,k$ be three integers, where $n,k\geq 1$. We begin with the definition of a Higgs coherent system of type $(n,d,k)$ on $X$. 
\begin{defi} 
\begin{em}
A {\bf \color{black}Higgs coherent system of type $(n,d,k)$} on $X$ is a pair $(E,V)$ consisting of a vector bundle $E$ on $X$ of rank $n$ and degree $d$, and a linear subspace $V$ of $H^0(X,\End(E)\otimes \KK)$ of dimension $k$. We will denote by {\bf \color{black}$H(n,d,k)$} the set of all Higgs coherent systems of type $(n,d,k)$ on $X$. 
\end{em}
\end{defi} 
We refer to the space $H^0(X,\End(E)\otimes \KK)$ as the twisted endomorphisms of $E$, or as the space of Higgs fields of $E$. The pair $(0, 0)$ is referred to as the trivial Higgs coherent system and sometimes will be denoted by $0: =(0, 0)$.

Let $E$ be a vector bundle of type $(n,d)$. By Serre's duality and Riemann-Roch's Theorem, we have that \begin{equation}\label{DimEndTor}
h^0(\End(E)\otimes \KK)=n^2(g-1)+h^0(\End(E)).
\end{equation}
Consequently, the dimension of the twisted endomorphisms of $E$ depends on the dimension of the space of automorphisms of $E$. In particular, if $L$ is a line bundle, it follows that $\End(L)\ox \KK=\KK$. Then, the space of twisted endomorphism of a line bundle is the space of holomorphic $1$-forms $H^0(X,\KK)$. In particular, $H(1,d,k)\neq \emptyset$ if and only if $k\leq g$. 

From now on, we assume that $n\geq 2$. The next result is related to the non-emptiness of $H(n,d,k)$. 

\begin{prop}\label{notempty}
For any type $(n,d,k)$, the set of Higgs coherent systems $H(n,d,k)$ is nonempty. Moreover, let $(E,V)$ be a Higgs coherent system in $H(n,d,k)$:
\begin{itemize}
\item[i.] if $k>n^2(g-2)+1$, then $E$ is not simple
\item[ii.] if $k>n^2(g-\frac{1}{2})-\frac{n}{2}+1$, then $E$ is decomposable or unstable
\end{itemize} 
\end{prop}
\begin{proof}
Let $E$ be vector bundle of type $(n,d)$, by (\ref{DimEndTor}) we have that 
$$h^0(\End(E)\otimes\KK)\geq n^2(g-1)+1.$$
If $k\leq n^2(g-1)+1$, always exists a linear subspace of dimension $k$ of $H^0(X,\End(E)\otimes\KK)$, and thus $H(n,d,k) \ne \emptyset.$ Now, if $k>n^2(g-1)+1,$ we can always produce a decomposable vector bundle $E$ such that $h^0(\End(E))$, and hence $h^0(\End(E)\ox \KK)$, is big enough. We proceed as follows: Let $r$ be an integer, and $x\in X$ a point, we consider the vector bundle of type $(n,d)$
$$E_r:=\mathcal{O}_X(-rx)\oplus \mathcal{O}_X((d+r)x)\oplus \mathcal{O}_X^{\oplus(n-2)}.$$
We have that
{\small \begin{align*}
\End(E_r)=&\mathcal{O}_X((d+2r)x)\oplus \mathcal{O}_X^{\oplus n-2}((d+r)x)\oplus\mathcal{O}_X^{\oplus n-2}(rx)\oplus \mathcal{O}_X^{\oplus (n-2)^2+2}\oplus \\
&\mathcal{O}_X((-d-2r)x) \oplus \mathcal{O}_X^{\oplus n-2}((-d-r)x)\oplus \mathcal{O}_X^{\oplus n-2}(-rx).
\end{align*}} 

Take $r>\max\{0,-d\}$, it follows that
{\small \begin{align*}
h^0(\End(E_r))&=(n-2)^2+2+h^0(\mathcal{O}_X((d+2r)x)+(n-2)[h^0(\mathcal{O}_X((d+r)x)+h^0(\mathcal{O}_X(rx))]\\
               &\geq (n-2)^2+2+ h^0(\mathcal{O}_X((d+2r)x)\\
               &\geq (n-2)^2+d+2r+3-g
\end{align*}} 

Therefore,
\begin{align*}
h^0(\End(E_r)\otimes \KK)&=n^2(g-1)+h^0(\End(E_r))\\
                          &\geq n^2(g-1)+(n-2)^2+d+2r+3-g\\
                          &\geq n^2(g-1)+d-g+2r.                    
\end{align*}
%
%
Take $r>\max\{0,-d,\frac{1-d+g}{2},\frac{k-n^2(g-1)-d+g}{2}\}$. Then, $h^0(\End(E_r)\otimes \KK)\geq k$, and hence there is a subspace $V_r$ of dimension $k$ of $H^0(X,\End(E_r)\otimes \KK)$, which implies that $(E_r,V_r)$ lie in $H(n,d,k)$. For the last part, if $E$ is a simple simple vector bundle, then $k\leq h^0(\End(E)\otimes \KK)=n^2(g-1)+1$. If $E$ is an indecomposable and semistable vector bundle, then $h^0(\End(E))\leq 1+\frac{n(n-1)}{2}$ (\cite[Proposition 1]{leticia}). Therefore, $k\leq h^0(\End(E)\ox \KK)\leq n^2(g-\frac{1}{2})-\frac{n}{2}+1$, which completes the  proof.
\end{proof}

\subsection{Morphisms of Higgs coherent systems}
Now we define a morphism between Higgs coherent systems. For this, first note that a morphism of vector bundles, $\alpha \colon F \to E$, induces the following morphisms
\begin{align*}
\alpha \otimes id_{F^\vee\otimes \KK} &\colon F \otimes F^\vee\otimes \KK \to E\otimes F^\vee\otimes \KK ~\text{and}\\
\alpha^\vee\otimes id_{E\otimes \KK}  &\colon E^\vee\otimes E\otimes \KK \to F^\vee\otimes E\ox \KK,
\end{align*}
and thus, we have the induced homomorphisms on the cohomology groups 
\begin{equation}\label{diagL}
\xymatrix@R=10mm{
                             &H^0(X,\End(E)\otimes K)\ar[d]^{\alpha_2}\\
H^0(X,\End(F)\otimes K)\ar[r]^{\alpha_1} &H^0(X,F^\vee\ox E\otimes K),
}
\end{equation}
where
\begin{align*}
\alpha_1(\psi)&=H^0(\alpha\ox id_{F^\vee\ox \KK})(\psi)=(\alpha\ox id_{F^\vee\ox \KK})\circ \psi, \text{~~~~for all}~\psi\in H^0(X,\End(F)\ox \KK)~\text{and} \\
\alpha_2(\varphi)&=H^0(\alpha^\vee \ox id_{E\ox \KK})(\varphi)=(\alpha^\vee \ox id_{E\ox \KK})\circ \varphi,\text{~~~~~~~~~~~~~~~~~~~~~~~~for all}~\varphi\in H^0(X,\End(E)\ox \KK).
\end{align*}
The above two homomorphisms share a common codomain, which allows us to relate subspaces of $H^0(X,\End(F)\otimes K)$ to subspaces of $H^0(X,\End(E)\otimes K)$. We use this fact to define a morphism between Higgs coherent systems.

\begin{defi}\label{schmorphism}
\begin{em}
A {\bf \color{black}morphism} of Higgs coherent systems, $h_\alpha \colon (F,W)\to (E,V)$, is a morphism $\alpha\colon F \to E$ such $\alpha_1(W)\subseteq \alpha_2(V)$. 
\end{em}
\end{defi}


\begin{obs}\label{obsuni}
\begin{em}
Let $(F,W)$ and $(E,V)$ be two Higgs coherent systems.
\begin{itemize}
\item[i.]  A morphism from the trivial Higgs coherent system $(0,0)$ to a Higgs coherent system $(E,V)$, by convention, consists of the inclusion in each component: $0\subset E$ and $0\subset V$. The morphism $(E,V) \to (0,0)$ means that $E$ maps to $0$ and $V$ maps to $0$
\item[ii.] Let $h_\alpha \colon (F,W)\to (E,V)$ be a morphism of Higgs coherent systems. For any $\psi\in H^0(X,\End(F)\otimes \KK)$, there exists $\varphi\in H^0(X,\End(E)\otimes \KK)$ such that $\alpha_1(\psi)=\alpha_2(\varphi)$. Since $H^0(X,\End(E)\ox \KK)=\Hom(E,E\ox \KK)$, any section $\varphi$ of $\End(E)\ox\KK$ correspond to a morphism $\varphi\colon E\to E\ox \KK$. Therefore, $\alpha_1(\psi)=\alpha_2(\varphi)$ if and only if the following diagram commutes: 
\begin{equation}\label{diag}
\xymatrix{
F\ar[rr]^{\alpha}\ar[d]_{\psi}\ar@{}[drr(.9)]|-{\circlearrowleft}& & E\ar[d]^{\varphi}
\\
F\otimes \KK \ar[rr]_{\alpha\otimes id_\KK}& &E\otimes \KK
}
\end{equation}


\item[iii.] A morphism $h_\alpha\colon (F,W)\to (E,V)$ is determined by the morphisms $(\alpha, \alpha_1, \alpha_2)$. In particular, if $\alpha\colon F\to E$ is injective, then $\alpha_1$ is injective. Analogously, if $\alpha\colon F\to E$ is surjective, then $\alpha_2$ is injective 
\item[iv.]  Let $h_\alpha \colon (F,W)\to (E,V)$ be a morphism of Higgs coherent systems such that $\alpha$ is surjective. The injectivity of $\alpha_2$ restricted to $V$, along with the condition \mbox{$\alpha_1(W)\subseteq \alpha_2(V)$} imply that the map 
\begin{align}\label{hatalfa}
\hat{\alpha}\colon H^0(X,\End(F)\ox \KK) &\dashrightarrow H^0(X,\End(E)\ox \KK)\\
\psi&\mapsto \alpha_2^{-1}(\alpha_1(\psi)),\nonumber
\end{align}
is well defined on $W$ and $\hat{\alpha}(W)\subseteq V$. In this case, $h_{\alpha}$ is written as $(\alpha, \alpha_1, \alpha_2,\hat{\alpha}),$ or simply $(\alpha,\hat{\alpha})$
\end{itemize}
\end{em}
\end{obs}

Let $h_\alpha\colon (F,W)\to (E,V)$ be a morphism of Higgs coherent systems. If $\alpha\colon F\to E$ is an isomorphism then $\alpha_1$ and $\alpha_2$ are injective, and thus the homomorphism in ($\ref{hatalfa}$) is well-defined for all Higgs field of $F$. We use these fact to define when the Higgs coherent systems $(F,W)$ and $(E,V)$ are isomorphic.

\begin{defi}\label{Defiso}
\em
Two Higgs coherent systems, $(F,W)$ and $(E,V)$, are {\bf \color{black}isomorphic systems} if there exists a morphism $h_\alpha \colon (F,W)\to (E,V)$ with $\alpha$ an isomorphism such that $\hat{\alpha}(W)=V$. In this case, we say that $h_\alpha$ is an {\bf \color{black}isomorphism} and we write $(F,W)\sim(E,V).$
\end{defi}

\begin{obs}
\em
Let $h_\alpha=(\alpha,\hat{\alpha})\colon (F,W)\to (E,V)$ be an isomorphism, from (\ref{hatalfa}), it follows that $\hat{\alpha}$ is an isomorphism and
\begin{equation}\label{finvectors}
\begin{split}
\hat{\alpha}=H^0(\alpha\ox\alpha^{-\vee}\ox id_{\KK})\colon H^0(X,\End(F)\otimes \KK)&\to H^0(X,\End(E)\otimes \KK)\\
\psi&\mapsto (\alpha\otimes \alpha^{-\vee}\ox id_\KK)\circ \psi. 
\end{split}
\end{equation}
In particular, $(F,W)$ and $(E,V)$ have the same type.
\end{obs}

\subsection{Families of Higgs coherent systems}
In what follows, we introduce the notion of how Higgs coherent systems deform. That is, we define the concept of families with which we will work. Roughly speaking, a family of Higgs coherent systems parameterized by a scheme $S$ is a pair $(\E,\V)$ where $\E$ is a family of vector bundles parameterized by $S$ and $\V$ is a sub-bundle of rank $k$ of a sheaf on $S$, whose fiber at each $s$ in $S$ is precisely $H^0(X,\End(\E_s)\otimes \KK)$.

Let $\mathcal{E}$ be a family of vector bundles of type $(n,d)$ parameterized by a scheme $S$. For any point $s$ in $S$, we are interested in the subspaces of $H^0(X,\End(\E_s)\otimes \KK)$, so it is natural to consider the direct image $R^0\pi_{2^*}(\End(\mathcal{E})\ox \pi_X^*(\KK))$, but in general the fibers of this sheaf are not equal to $H^0(X,\End(\E_s)\otimes \KK)$ (see \cite[Corollary 4.2.9]{potier}). Instead of $R^0\pi_{2*}(\End(\mathcal{E})\ox \pi_X^*(\KK))$, we consider the first direct image $R^1\pi_{2^*}(\End(\mathcal{E}))$. That is, 
$$\xymatrix{
\End(\mathcal{E})\ar[d]&  R^1\pi_{2^*}(\End(\mathcal{E}))\ar[d]
\\
X\times S\ar[r]^{\pi_2}& S.
}$$
We define $\mathcal{R}^{1^\vee}_{\mathcal{E}}:=R^1\pi_{2^*}(\End(\mathcal{E}))^\vee$. The following result proves that a locally free sub-sheaf of $\mathcal{R}^{1^\vee}_{\mathcal{E}}$ can be regarded as a family of subspaces of twisted endomorphisms of the fibers of $\mathcal{R}^{1^\vee}_{\mathcal{E}}$.

\begin{lema} \label{fibre}
Let $\V$ be a sub-bundle of rank $k$ of $\mathcal{R}^{1^\vee}_{\mathcal{E}}$. For any $s$ in $S$, it follows that $\V_s$ is a $k$-dimensional subspace of $H^0(X,\End(\E_s)\otimes \KK)$.
\end{lema}
\begin{proof}
Let $\V$ be a locally free subsheaf of rank $k$ of $\mathcal{R}^{1^\vee}_{\mathcal{E}}$ and $s$ a point in $S$. It follows that $R^2\pi_{2^*}(\End(\E))=0$, because $X$ is a curve. Since the projection $\pi_S$ is proper, by \cite[Theorem 12.11]{hartshorne}  or \cite[Corollary 3]{mumfordabe}, we have that 
$$(R^1\pi_{2^*}(\End(\E)))_s= H^1(X,\End(\E_s)).$$
Now, \cite[Proposition 6.8]{hartshorne} and Serre's Duality imply that
\begin{align*}
(\mathcal{R}^{1^\vee}_{\mathcal{E}})_s&=\Hom(R^1\pi_{2*}(\End(\mathcal{E})),\mathcal{O}_S)_s\\
                             &=\Hom(H^1(X,\End(\E_s)),\mathbb{C})\\
                             &=H^0(X,\End(\E_s)\otimes \KK)
\end{align*}
Therefore, for all $s$ in $S$, $\V_s$ is a $k$-dimensional subspace of $H^0(X,\End(\E_s)\otimes \KK)$, as claimed.
\end{proof}
Let $\E$ and $\V$ as in the above result, for all $s$ in $S$, the pair $(\E_s,\V_s)$ is a Higgs coherent system of type $(n,d,k)$. This motives the following definition.
\begin{defi}
\begin{em}
A {\bf \color{black}family of Higgs coherent systems} of type $(n,d,k)$ parameterized by a scheme $S$ is a pair $(\mathcal{E},\mathcal{V})$, where  $\mathcal{E}$ is a family of vector bundles of type $(n,d)$ parameterized by $S$ and $\mathcal{V}$ is a sub-bundle of rank $k$ of $\mathcal{R}^{1^\vee}_{\mathcal{E}}$. Two families parameterized by $S$, $(\mathcal{E},\mathcal{V})$ and $(\mathcal{E}',\mathcal{V}')$, are said to be {\bf \color{black}equivalent} if, for all point $s$ in $S$, the Higgs coherent systems $(\mathcal{E}_s,\mathcal{V}_s)$ and $(\mathcal{E}'_s,\mathcal{V}'_s)$ are isomorphic. In this case, we write $(\mathcal{E},\mathcal{V})\approx (\mathcal{E}',\mathcal{V}')$.
\end{em}
\end{defi}

To define the notion of induced family by a morphism of schemes, we will use the following result.

\begin{lema}\label{flat_to_pull}
Let $(\E,\V)$ be a family of Higgs coherent systems of type $(n,d,k)$ parameterized by $S$. Let $\phi\colon S'\to S$ be a flat morphism. The pair $(\phi^*(\mathcal{E}),\phi^*(\mathcal{V}))$ is a family of Higgs coherent systems of type $(n,d,k)$ parameterized by $S'$.
\end{lema}
\begin{proof}
By definition, $\phi^*(\mathcal{E})=(id_X\times \phi)^*(\mathcal{E})$, which is a family of vector bundles of type $(n,d)$ parameterized by $S'$. It is enough to prove that $\phi^*(\mathcal{V})$ is a sub-sheaf of $\R^{1^\vee}_{\phi^*(\E)}$.
Since pull-back commutes with tensor product, and using the projection formula (\cite[Proposition 2.32]{Qing}), we have that 
$$\R^{1^\vee}_{\phi^*(\E)}=R^1\pi_{2^*}(\End(\phi^*(\E)))^\vee=R^1\pi_{2^*}(\phi^*(\End(E)))^\vee=(\phi^*(R^1\pi_{2^*}(\End(E))))^\vee.$$
It is known that pull-back of a coherent sheaf under a flat morphism commutes with taking dual \cite[Proposition 1.8]{hartReflexive}, hence
\begin{align*}
\R^{1^\vee}_{\phi^*(\E)}&=\phi^*(R^1\pi_{2^*}(\End(E)))^\vee)\\
				        &=\phi^*(\R^{1^\vee}_\E)
\end{align*}
Finally, from the inclusion $0\to\V\to \R^{1^\vee}_\E$, since pull-back by a flat morphism preserve injectivity (\cite[Proposition 14.11]{urlich}), it follows that \mbox{$0\to\phi^*(\V)\to \phi^*(\R^{1^\vee}_\E)=\R^{1^\vee}_{\phi^*(\E)}$}, which completes the proof.
\end{proof}

Let $\phi\colon S'\to S$ be a morphism and let $(\mathcal{E},\mathcal{V})$ be a family of Higgs coherent systems of type $(n,d,k)$ parameterized by $S$, we define {\bf the induced family by $\phi$}, denoted as $\phi^*(\mathcal{E},\mathcal{V})$, to be $((id_X\times \phi)^*(\mathcal{E}),\phi^*(\mathcal{V}))$. 

Note that to give a family of subspaces of the spaces of Higgs fields of dimension $k$ is equivalent to give a family of quotients of dimension $k$ of the dual space of the space of Higgs fields. Therefore, a family of Higgs coherent systems of type $(n,d,k)$ parameterized by a scheme, $(\E, \V)$, is equivalent to a pair $(\E,\mathcal{Q})$, where $\mathcal{Q}$ is a locally free quotient of rank $k$ of $R^1\pi_{2^*}(\End(\mathcal{E}))$. Since the pull-back is a right-exact functor, the assumption of flatness on the morphisms in Lemma \ref{flat_to_pull} can be dropped if, as family, we consider $(\E,\V^\vee)$ instead of $(\E,\V)$. Under this equivalence, the notion of induced family satisfies the following functorial properties (cf. \cite[Conditions 1.4 (c)]{newstead}), which allows to define the moduli problem for Higgs coherent systems.

\begin{obs}\label{induceproper}
\begin{em}
Let $\phi'\colon S''\to S'$ and $\phi\colon S'\to S$ be two morphisms of schemes and $(\mathcal{E},\mathcal{V})$ and $(\mathcal{E}',\mathcal{V}')$ be two families parameterized by $S$. Then
\begin{itemize}
\item[i.] $(\phi \circ \phi')^{*}((\mathcal{E},\mathcal{V}))=\phi '^{*} \circ \phi ^{*}((\mathcal{E},\mathcal{V}))$ 
\item[ii.]$id_S^{*}((\mathcal{E},\mathcal{V}))= (\mathcal{E},\mathcal{V})$
\item[iii.]$(\mathcal{E},\mathcal{V})\approx \mathcal{(\mathcal{E}',\mathcal{V}')} \Rightarrow \phi^{*}\mathcal{(\mathcal{E},\mathcal{V})} \approx \phi^{*}(\mathcal{E}',\mathcal{V}')$
\end{itemize} 
\end{em}
\end{obs}

We define the moduli problem of Higgs coherent systems of type $(n,d,k)$ over $X$ as follows:
\begin{equation}\label{ModProblem}
\fbox{$(H(n,d,k),\sim, (\E,\V),\approx).$}
\end{equation}
To construct the moduli space for the above problem, it is necessary to endow $H(n,d,k)/\sim$ with a geometric-algebraic structure that reflects how Higgs coherent systems deform in families. In the next section, we construct a projective scheme which is a fine moduli space for the moduli problem of Higgs coherent systems whose underlying vector bundles are stable.
\section{Moduli of Higgs coherent systems with stable underlying vector bundles}\label{Sec3}
In this section, we study Higgs coherent systems with underlying stable vector bundles and construct the fine moduli space for the corresponding moduli problem. 

Let $(E,V)$ be a Higgs coherent system of type $(n,d,k)$, we define
$$H_E(V)=\{V'\in \Grass(k,H^0(X,\End(E)\ox\KK)~|~(E,V)\sim(E,V')\}.$$
The next result is related to the isomorphism classes of Higgs coherent systems $(E,V)$, where $E$ is a simple vector bundle.
\begin{prop}\label{simpleclass}
Let $(E,V)$ be a Higgs coherent systems of type $(n,d,k)$ such that $E$ is a simple vector bundle. Then, 
$H_E(V)=\{V\}.$
\end{prop}
\begin{proof}
Let $V'$ a subspace of $H^0(X,\End(E)\ox\KK)$ and suppose that $(E,V)\sim (E,V')$. There exists an isomorphism $\alpha\colon E \to E$ such that $\hat{\alpha}(V)=V'$. Since $E$ is simple, there exits $c\in \bC^*$ such that $\alpha=c\cdot id_E$. It follows that $\hat{\alpha}(V)=(\alpha\otimes \alpha^{-\vee}\ox id_\KK)(V)=(c\cdot id_E\otimes c^{-1}\cdot id_{E^\vee}\ox id_\KK)(V)=V$. Therefore, $V=V'$, which completes the proof.  
\end{proof}

\begin{obs}\label{simpleiso}
\em
Let $(E,V)$ and $(E',V')$ to Higgs coherent systems of type $(n,d,k)$ such that $E$ and $E'$ are simple vector bundles. From Proposition \ref{simpleclass} and the fact that if $E\cong E'$, then $H^0(X,\Hom(E,E'))=\mathbb{C}$, it follows that $(E,V)\sim(E,V')$ if and only if $V=V'$. The equality $V=V'$ means that $V$ and $V'$ are equal as subspaces of $H^0(X,E^\vee\otimes E'\otimes \KK)$.
\end{obs}

Fix a possible type $(n,d,k)$ with $(n,d)=1$ and $k\leq n^2(g-1)+1$. We denote by $SH(n,d,k)$ the subset of $H(n,d,k)$ consisting of Higgs coherent systems $(E,V)$ of type $(n,d,k)$ such that $E$ is a stable vector bundle. 
Before we state and prove our main result, we recall some properties of the moduli space of stable vector bundles. 

Let $M^s(n,d)$ denote the moduli space of stable vector bundles of type $(n,d)$ over $X$. It is a smooth projective variety of dimension $n^2(g-1)+1$ (\cite[Remark 5.9]{newstead}). Moreover, its tangent bundle $T M^s(n,d)$ is a locally free sheaf of rank $n^2(g-1)+1$, and for any vector bundle $E$ in $M^s(n,d)$, it follows that $T_E M^s(n,d)=H^1(X,\End(E))$ (\cite[Theorem 2.7 and Corollary 2.8]{hartshornedefor}). We have the following result.

\begin{teo} \label{ParaStable} 
The $k$-th Grassmannian bundle $\pi\colon \Grass(k,\Omega^1 M^s(n,d))\to M^s(n,d)$ is in one-to-one correspondence with $SH(n,d,k)/\sim$. In particular, the space of isomorphism classes of Higgs coherent systems $SH(n,d,k)$ is a smooth irreducible projective scheme of dimension $(k+1)(n^2(g-1)+1)-k^2$. 
\end{teo}
\begin{proof}
The correspondence is consequence of the above properties of $M^s(n,d)$ and Proposition \ref{simpleclass}. For the last part, since $M^s(n,d)$ is a smooth, irreducible and projective, the $k$-th Grassmannian bundle $\Grass(k,\Omega^1 M^s(n,d))$ is a smooth irreducible projective scheme. Moreover, its dimension is equal to $\Dim(M^s(n,d))+\Dim(\Grass(k,\Omega^1_E M^s(n,d)))$, where $E$ is a stable vector bundle. Therefore, 
$$\Dim(\Grass(k,\Omega^1 M^s(n,d)))=(k+1)(n^2(g-1)+1)-k^2,$$
which completes the proof.
\end{proof}

Before proving that $\Grass(k,\Omega^1 M^s(n,d))$ is a fine moduli space, first we construct a family  $(\hat{\E},\mathcal{V})$ parameterized by $\Grass(k,\Omega^1 M^s(n,d))$. Since $(n,d)=1$, there is a universal vector bundle $\E$ over $X\times M^s(n,d)$ (\cite[Theorem 5.12]{newstead}). For the universal vector bundle, $\E$, we consider 
$$\R^{1^\vee}_{\E}=\Hom(R^1\pi_{2*}(\End(\mathcal{E})),\mathcal{O}_{M^s(n,d)}).$$ From the universal property of $\E$, it follows that $\R^{1^\vee}_{\E}=\Omega^1M^s(n,d)$ (\cite[Appendix 2.C.II and Theorem 10.2.1]{huybrechts}). 

We consider the pair $(\hat{\E} ,\V)$, defined as follows:
\begin{itemize}
\item[•] $\hat{\E}$ denote the pull-back $(id_X\times \pi)^*(\E)$
\begin{equation*}\label{Stable1}
\xymatrix{
\hat{\E}\ar[d]& &\E\ar[d]\\
X\times \Grass(k,\Omega^1 M^s(n,d))\ar[rr]^-{id_X\times \pi}& &X\times M^s(n,d)
}
\end{equation*}\label{Stable2}
\item[•] $\mathcal{V}$ denote the tautological sub-bundle of rank $k$ of $\pi^*(\Omega^1 M^s(n,d))$. The tautological sub-bundle fits in the universal exact sequence 
\begin{equation}
\xymatrix{
0\ar[r] & \mathcal{V}\ar[r] & \pi^*(\Omega^1 M^s(n,d))\ar[r]\ar[d] & \mathcal{Q}\ar[r] & 0\\
        &                   & \Grass(k,\Omega^1 M^s(n,d)),
}
\end{equation} 
where $\mathcal{Q}$ denote the tautological quotient.
\end{itemize}
 
\begin{lema}\label{familyStable}
The pair $(\hat{\E},\V)$ is a family of Higgs coherent systems in $SH(n,d,k)$ parameterized by $\Grass(k,\Omega^1 M^s(n,d))$.
\end{lema}
\begin{proof}
By construction, for all Higgs coherent systems $(E,V)$ in $SH(n,d,k)$, the restriction of $\hat{\E}$ to \mbox{$X\times \{(E,V)\}$} is isomorphic to $\E_E$, since $\E$ is a family of stable vector bundles of type $(n,d)$, we have that $\hat{\E}$ is also a family of stable vector bundles of type $(n,d)$.
From the commutative diagram
$$\xymatrix{
X\times \Grass(k,\Omega^1 M^s(n,d))\ar[rr]^-{id_X\times \pi}\ar[d]_{\pi_2}\ar@{}[drr(.99)]|-{\circlearrowleft}& &X\times M^s(n,d)\ar[d]^{\pi_2}\\
\Grass(k,\Omega^1 M^s(n,d))\ar[rr]_-\pi& &M^s(n,d),
}
$$
and the projection formula, we have that
\begin{align*}
R^1\pi_{2^*}(\End(\hat{\E}))& = R^1 \pi_{2^*}((id_X\times \pi)^*(\End(\E))) \\
                           & = \pi^*(R^1\pi_{2^*}(\End(\E)).
\end{align*}
Then,
\begin{align*}
\R^{1^\vee}_{\hat{\mathcal{E}}}= R^1\pi_{2^*}(\End(\hat{\E}))^\vee &= \pi^*(R^1\pi_{2^*}(\End(\E)))^\vee\\ 
                        &= \pi^*(R^1\pi_{2^*}(\End(\E))^\vee)\\
                        &= \pi^*(\mathcal{R}^{1^\vee}_\mathcal{E})\\
                        &=\pi^*(\Omega^1 M^s(n,d)).                      
\end{align*}
Therefore, $\V\subset\pi^*(\Omega^1 M^s(n,d))=\R^{1^\vee}_{\hat{\mathcal{E}}}$ is a locally free sheaf of rank $k$, and thus $(\hat{\E},\V)$ is a family of Higgs coherent systems in $SH(n,d,k)$, as claimed.
\end{proof}

In the next result, we will prove that the above family satisfies the universal property and thus there exists a fine moduli space for the moduli problem \mbox{$(SH(n,d,k),\sim, (\E,\mathcal{V}),\approx)$.}

\begin{teo}\label{ModuliHcsstable}
The $k$-th Grassmannian bundle, $\Grass(k,\Omega^1 M^s(n,d))$, together with the family $(\hat{\E},\mathcal{V})$ is a fine moduli space for the moduli problem \mbox{$(SH(n,d,k),\sim, (\E,\mathcal{V}),\approx)$}. Moreover, it is a smooth irreducible projective variety of dimension $(k+1)(n^2(g-1)+1)-k^2$.
\end{teo}
\begin{proof} The Grassmannian bundle $\Grass(k,\Omega^1 M^s(n,d))$ is in one-to-one correspondence with the space of isomorphism classes of Higgs coherent systems in $SH(n,d,k)$ (Theorem \ref{ParaStable}), and the pair $(\hat{\E},\mathcal{V})$ is a family parameterized by $\Grass(k,\Omega^1 M^s(n,d))$ (Proposition \ref{familyStable}). It is enough to prove that $(\hat{\E},\mathcal{V})$ satisfies the universal property.  To do this, note that, from the universal property of $\E$ and the definition of $\V$, we have that restriction of $(\hat{\E},\mathcal{V})$ to \mbox{$X \times \{(E,V)\}$} is isomorphic to $(E,V)$. We claim that the universal properties of $\E$ and $\V$, implies that $(\hat{\E},\mathcal{V})$ satisfies the universal property. Indeed, let $(\E',\V')$ be a family of Higgs coherent systems in $SH(n,d,k)$ parameterized by $T$. By the universal property of $\E$, there exists a unique morphism $\phi\colon T\to M^s(n,d)$ such that  $\E'$ is equivalent to $\phi^*(\E)$. Moreover, by the universal property of $\V$, there exist a unique morphism $\hat{\phi}\colon T\to \Grass(k,\Omega^1 M^s(n,d))$, such that $\hat{\phi}^*(\V)=\V'$ (\cite[Proposition 8.17(1)]{urlich}). We have that
$$\hat{\phi}^*(\hat{\E})=(id_X\times \hat{\phi})^*((id_X\times\pi)^*(\E))=(id_X\times (\pi\circ \hat{\phi}))^*(\E)=(\pi\circ \hat{\phi})^*(\E).$$
Note that, for all $t\in T$, $\V'_t=\hat{\phi}^*(\V)_t=\V_{\hat{\phi}(t)}=\hat{\phi}(t)$ and $\hat{\phi}^*(\hat{\E})_t=(\pi\circ\hat{\phi})^*(\E)_t=\E_{\pi(\hat{\phi}(t))}=\E_{\pi(\V'_t)}=\E '_t$. 
Therefore, $\hat{\phi}^*(\hat{\E},\mathcal{V})$ is equivalent to $(\E',\V')$, as claimed. 
\end{proof}

\begin{obs}
\em
It is known that if $(n,d)\neq 1$, the universal family of stable vector bundles $\E$ does not exist, that is a fine moduli space for stable vector bundles does not exist (\cite[Theorem 2]{ramanan}). For this case, such a universal family exists locally in the étale topology (\cite[Theorem 1.21. (4)]{simpson}). Since sheaves in the Zariski topology can be consider as sheaves in the étale topology \cite[Proposition 4.15]{knutson}, for any type $(n,d,k)$ with $k<n^2(g-1)+1$, we can proceed locally with the above arguments to obtain a local version of Theorem \ref{ModuliHcsstable}.
\end{obs}

In \cite{oscar}, the authors study Higgs coherent systems of type $(1,d,k)$. They consider notions of families and equivalence between them that differ from the corresponding notions treated in this work. In particular, for the concept of a family they consider a zero direct image instead of the first direct image as in our case (see \cite[Definición 3.3.]{oscar}). Nevertheless, their definition has a solvable issue: in general, for a family $(\E,\V)$ parameterized by $T$ and a point $t$ in $T$. it is not necessary true that $(\E,\V)_t$ is a Higgs coherent system (cf. \cite[Corollary 4.2.9]{potier}), as claimed in \cite[Proposición 3.1.1 (a)]{oscar}. This issue can be solved, for example, by asking for the extra condition of injectivity, as in coherent systems (cf. \cite[Definition 1.11]{ragha}), on the natural morphism between the fiber of the zero direct image and the space of twisted endomorphisms. 

We now apply our approach to construct the moduli space of Higgs coherent systems of type $(1,d,k)$. Let $\Pic^d(X)$ be the Picard variety of $X$ and $\LL^d$ its Poincaré bundle. It is known that the tangent space at a line bundle $L$ in $\Pic^d(X)$ is $H^1(X,\Oc_X)$. Since $\Pic^d(X)$ is a smooth variety, its tangent sheaf is locally free and has rank $g$. Moreover, since $\Pic^d(X)$ is isomorphic to $\Pic^0(X)$, which is an abelian variety, its tangent bundle is trivial (\cite[Pag. $42$ (iii)]{mumfordabe}). By the universal property of $\mathcal{L}^d$, it follows that \mbox{$\R^{1^\vee}_{\mathcal{L}^d}=\Omega^1\Pic^d(X)=H^0(X,\KK)\ox \Oc_{\Pic^d(X)}$} (\cite[Theorem 10.2.1]{huybrechts}), and hence the $k$-th Grassmannian bundle $\Grass(k,\Omega^1\Pic^d(X))$ is the product $\Pic^d(X)\times \Grass(k,H^0(X,\KK))$. In particular, from Theorem \ref{ModuliHcsstable}, for Higgs coherent systems of type $(1,d,k)$, with $k\leq g$, we obtain the following result (cf. \cite[Teorema 3.2 ]{oscar}).
\begin{cor}\label{ModuliHcs1dk}
The moduli space of Higgs coherent systems of type $(1,d,k)$ is $\Pic^d(X)\times \Grass(k,H^0(X,\KK))$ together with the family $(\hat{\LL^d},\mathcal{V})$. Moreover, it is a smooth irreducible projective variety of dimension $g(k+1)-k^2$.
\end{cor}

{\bf Acknowledgments.} This paper has its origins in the PhD thesis of the author \cite{Ich0}, the research problem of which was proposed to the author by Professor Leticia Brambila-Paz. I thank Professor Alexander Quinteto Vélez and Professor Ronald A. Zúñiga-Rojas for their enlightening discussions, which significantly contributed to this work.

\bibliographystyle{alpha}
\bibliography{References}

\end{document}